\documentclass[11pt]{article}
\usepackage{amsmath,amsfonts, amsthm, amssymb}
\usepackage{graphicx}
\usepackage{microtype}

\def\Real{{\mathbf R}}
\def\bbZ{{\mathbf Z}}
\def\Sphere{{{\mathbb S}}}
\def\bbN{{{\mathbb N}}}

\def\bbS{{\mathbb S}}

\newcommand{\mcal}[1]{{\mathcal{#1}}}

\newtheorem{theorem}{Theorem}[section]
\newtheorem{question}{Question}
\newtheorem{lemma}[theorem]{Lemma}
\newtheorem{conjecture}[theorem]{Conjecture}
\newtheorem{definition}[theorem]{Definition}

\title{Towards a refined estimate for topological degree
in one dimension}
\author{Felipe Hern{\'a}ndez}
\date{\today}

\begin{document}
\maketitle

\begin{abstract}
In this paper we prove an inequality inspired by a
conjecture of Brezis, which asks for a bound for the topological degree of a
map from the circle to itself in terms of a nonlocal integral.
\end{abstract}

\section{Introduction}
Thie note concerns a question posed by Haim Brezis regarding the
topological degree of continuous maps from the circle $\Sphere^1$ to
itself~\cite[Open Problem 3]{brezis2006new}.

\begin{conjecture}
\label{conj:main-Q}
There exists a constant $C>0$ and $\delta_0>0$
such that the following holds:
For any continuous function $f\in C^0(\Sphere^1;\Sphere^1)$
and $\delta<\delta_0$,
\begin{equation}
\label{eq:conjecture}
|\deg f| \leq C \delta \iint_{|f(s)-f(t)|>\delta} \frac{dsdt}{|s-t|^2}.
\end{equation}
\end{conjecture}

This conjecture was inspired by the work of
Bourgain, Brezis, and
Nguyen~\cite{bourgain2005new}, which established~\eqref{eq:conjecture} without the
factor of $\delta$ in the right hand side.
More recently, the analogous statement in dimension $d\geq 2$ has been proven by
Hoai-Minh Nguyen~\cite{nguyen2017refined}.  Remarkably, the one-dimensional case
remains the most difficult to handle.

In this note we prove a closely related statement which still seems to capture some
of the difficulty of Conjecture~\ref{conj:main-Q}.  For any map
$f\in C^0(\Sphere^1;\Sphere^1)$ we can find some continuous phase function
$\varphi:\Real\to\Real$
so that $f(e^{2\pi is}) = e^{2\pi i\varphi(s)}$ and $\deg f = \varphi(1)-\varphi(0)$.
The main result is phrased in terms of functions $\varphi$ defined in this way.
\begin{theorem}
\label{thm:main-thm}
Let $\varphi:\Real\to\Real$ be a continuous function such that there exists $D\in\Real$
such that $\varphi(x+1) = \varphi(x) + D$ for any $x\in\Real$.
Then for $\delta<0.01$,
\begin{equation}
\label{eq:main-bd}
|D| \leq 10^9 \delta \iint_{\{s,t\in[0,2] \mid |\varphi(s)-\varphi(t)|>\delta\}} \frac{dsdt}{|s-t|^2}.
\end{equation}
\end{theorem}

To prove Theorem~\ref{thm:main-thm}, we investigate a closely related problem in which
the domain of integration is restricted to $[0,1]$ rather than $[0,2]$, and so that no periodicity
conditions are imposed on the function.
\begin{question}
    Is it the case that there exists $C,\delta_0>0$ such that
    for all $g\in C^0([0,1];\Real^+)$ and $\delta<\delta_0$,
    \[
        |g(1)-g(0)| \leq C
        \delta\cdot\iint_{\{s,t\in[0,1] \mid |g(s)-g(t)|>\delta\}} \frac{dsdt}{|s-t|^2}?
    \]
    \label{flattened-q}
\end{question}
The answer to Question~\ref{flattened-q} is
negative: for each $\delta>0$ there exist examples of functions for which
$g(1)-g(0)$ grows as large as desired and for which the integral on the right
is as small as desired.   In fact, the function
$g(s) = \alpha \log(s+\beta)$ works for sufficiently small $\beta$ and
$\alpha$ depending on $\delta$.

The proof of Theorem~\ref{thm:main-thm} will
proceed by showing that all such counterexamples to Question~\ref{flattened-q}
have such a logarithmic-type singularity, either at $0$ or at $1$.  This will allow
us to recover the desired contribution from the integration for $s,t$ on either side of the
point $1$.

The main challenge is that tight logarithmic-type singularities appearing within
the interval $[0,1]$ do not contribute much to the nonlocal integral~\eqref{eq:main-bd}, so
there could be millions of such almost-singularities in $[0,1]$.
Therefore, although $\varphi$ is continuous, its value at any given point is almost meaningless
so instead the argument works more with the values of $\varphi$ on finite intervals.

\subsection{Acknowledgements}
The author is supported by the John and Fannie Hertz Foundation.

\section{The structure result}
The main result is actually a structure theorem for positive functions
$g\in C^0([0,1];\Real)$ that are negative examples for
Question~\ref{flattened-q}.
We fix the function $g$ and the scale $\delta$.
We suppose for simplicity that $M\delta = g(1)$ for some $100<M\in\bbN$,
and that $g(0)=0$.
We also assume without loss of generality that the level sets
$\{s \mid g(s) = k\delta\}$ for $k\in\bbN$ are finite.  If not, one may make
an extremely small perturbation to $g$ to ensure that this is the case.

To state the structure theorem we will need several definitions.
First we define the discretized level sets
\begin{equation}
    E_k := \{s\in[0,1] \mid g(s)\in [k\delta, (k+1)\delta]\}
    \label{ek-defn}
\end{equation}
and their neighbors
\begin{equation}
    N_k := E_{k-1}\cup E_k\cup E_{k+1}.
    \label{nk-defn}
\end{equation}

Using $E_k$ and $N_k$ we may write a simple lower bound for our nonlocal
integral quantity.
\begin{equation}
    \sum_{k\in \bbN}
    \iint_{s\in E_k; t\not\in N_k} \frac{dsdt}{|s-t|^2}
    \leq \iint_{|g(s)-g(t)|>\delta} \frac{dsdt}{|s-t|^2}.
\end{equation}
To simplify the study of the left hand side, we make the following
observation.  If there exists an interval $I\subset[0,1]$ such
that $|E_k\cap I| > 0.01|I|$ and $|N_k^c\cap I| > 0.01|I|$, then
\[
    \iint_{s\in E_k; t\not\in N_k} \frac{dsdt}{|s-t|^2}
    \geq 10^{-4}.
\]
This motivates us to define the set of satisfied indices,
\begin{equation}
    \mathcal{S} = \{k\in\mathbb{N}\mid
        \exists I=(a,b)\subset [0,1] \text{ s.t. }
        |E_k\cap I| > 0.01|I| \text{ and }
        |N_k^c\cap I| > 0.01|I|\},
        \label{sprime-defn}
\end{equation}
for which we have the bound
\begin{equation}
\label{S-int-bd}
\#\mcal{S} \leq 10^4 \iint_{|g(s)-g(t)|>\delta}\frac{dsdt}{|s-t|^2}.
\end{equation}

The definition of $\mathcal{S}$ can be rephrased in terms of the notions
of empty and full intervals, which we define now.
\begin{definition}[Empty and full intervals]
    Let $I\subset [0,1]$ be an interval and let $0<\alpha<1$.
    We say that $I$ is $\alpha$-full for the index $k\in\bbZ$ if
    \[
        |E_k\cap I| > \alpha |I|.
    \]
    On the other hand, we say that $I$ is $\alpha$-empty for $k$ if
    \[
        |N_k^c\cap I| > \alpha |I|.
    \]
\end{definition}
Observe that for $\alpha<0.5$, it is possible for an index to be
both $\alpha$-full and $\alpha$-empty for an interval.

The idea is to prove that either $\mathcal{S}$ is a large set, or else
$g$ has a special behavior near the endpoints $0$ and $1$.
There are several possible behaviors, corresponding to the
different counterexamples to Question~\ref{flattened-q}.

\begin{definition}[Logarithmic singularities]
    We say that a function $g\in C^0([0,1];\Real^+)$ has a
    \emph{logarithmic singularity at the boundary} if there
    exists an $m>0.1M$ and sequences of indices
    \[
        k_0 < k_1 < \cdots < k_m
    \]
    and lengths
    \[
        \tau_0 > 3\tau_1 > \cdots > 3^m\tau_m
    \]
    such that one of the following holds:
    \begin{itemize}
        \item \textbf{Growth near 1}: The interval
        $(1-\tau_i,1)$ is an $0.03$-full
        interval for the index $k_i$ and $(0,\tau_i)$ is an $0.03$-empty
        interval for the indices $k_i-M-1$, $k_i-M$, and $k_i-M+1$.
        \item \textbf{Growth near 0}: The interval
        $(0,\tau_{m-i})$ is an $0.03$-full interval for the index $k_i$
        and $(1-\tau_{m-i},1)$ is an $0.03$-empty interval for the
        indices $k_i+M-1$, $k_i+M$, and $k_i+M+1$.
    \end{itemize}
    A set $\{(k_i,\tau_i)\}_{i=1}^m$ satisfying one of these conditions
    is called a \emph{witness} to the logarithmic growth of $g$.
\end{definition}

We are now ready to state the structure result.
\begin{theorem}
    \label{structure-thm}
    At least one of the following holds.  Either
    \[
        \#\mathcal{S} \geq 0.01M
    \]
    or else $g$ has a logarithmic singularity as defined above.
\end{theorem}

\begin{proof}[Proof of Theorem~\ref{thm:main-thm} using
Theorem~\ref{structure-thm}]

We apply Theorem~\ref{structure-thm} to the function $\varphi$ restricted to the
interval $[0,1]$.  In the case that $\#\mathcal{S}\geq 0.01M$, then~\eqref{S-int-bd}
implies the result (since $M=\varphi(1)-\varphi(0)$).

We can therefore suppose the existence of a logarithmic singularity, and
without loss of generality we will suppose that there is growth at $1$,
meaning that there exist an increasing sequence $\{k_j\}_{j=1}^m$,
$m>0.1M$, and a geometrically decreasing sequence
$\{\tau_j\}_{j=1}^m$ such that $(1-\tau_i,1)$ is a $0.03$-full
interval for $k_i$ and $(0,\tau_i)$ is an $0.03$-empty interval for
$k_i-M-1,k_i-M$,and $k_i-M+1$.

Then, extending the notion of empty intervals to the function $\varphi$ on $[0,2]$, this implies
that $(1,1+\tau_i)$ is a $0.03$-empty interval for $k_i-1$,$k_i$, and $k_i+1$.
Therefore
\begin{equation}
\int_{s\in (1-\tau_i,1)\cap E_i}\int_{t\in(1,1+\tau_i)\cap N_i^c}
\frac{dsdt}{|s-t|^2} \geq 10^{-6}.
\end{equation}
Each $1\leq i\leq m$ contributes $10^{-6}$, so the result follows.
\end{proof}

\section{Proof of Theorem~\ref{structure-thm}}
Throughout the proof we will suppose that $g\in C^0([0,1];\Real^+)$ is
some fixed positive function with $g(0)=0$ and $M\delta =g(1)$,
and that $\#\mcal{S}\leq 0.01M$, where $\mcal{S}$ is defined
in~\eqref{sprime-defn}.

We first define the set $\mathcal{P}$ of pairs of unsatisfied indices
\begin{equation}
    \mathcal{P} =
    \{k\in\bbN\mid k\not\in \mathcal{S}, k-1\not\in \mathcal{S},
    \text{ and } k+1\not\in \mathcal{S}\}.
    \label{P-defn}
\end{equation}
Observe that $\#\mathcal{P}^c \leq 3\#\mathcal{S}$,
so we will be able to find plenty of indices in $\mathcal{P}$.  This
is useful because the condition $k\in\mathcal{P}$ is actually rather
strong.
\begin{lemma}
    If $k\in \mathcal{P}$ and $I$ is an interval with
    $0.01|I|\leq |E_k\cap I|\leq 0.1|I|$,
    then either $|E_{k-1}\cap I| > 0.8|I|$ or $|E_{k+1}\cap I| > 0.8|I|$.
    \label{rigidity-lemma}
\end{lemma}
\begin{proof}
    Since $k\not\in \mathcal{S}$, it follows that $|N_k\cap I| > 0.99 |I|$.
    Suppose that
    $|E_{k-1}\cap I| > 0.01|I|$.  Then, using the fact that
    $k-1\not\in \mathcal{S}$, we have $|E_{k+1}\cap I| < 0.01|I|$.
\end{proof}

We will now pick a very special set of intervals $\mathcal{I}_k$ which
we call canonical intervals.
\begin{definition}[Canonical intervals]
    \label{canonical-intervals-defn}
    We say that the collection of intervals
    $\mathcal{I}=\{\{(s_{i,k},t_{i,k})\}_{i=1}^{n_k}\}_{k=1}^M$
    are \emph{canonical} if for each $k\in\bbN$, the collection
    $\mathcal{I}$ satisfies the following conditions:
    \begin{itemize}
        \item \textbf{Covering}:
            \[
                E_k \subset
                \bigcup_{i=1}^{n_k} (s_{i,k},t_{i,k}).
            \]
        \item \textbf{Ordering}:
            For each $k\in[M]$,
            \[
                s_{1,k} < t_{1,k} < s_{2,k} < t_{2,k} < \cdots <
                s_{n_k,k} < t_{n_k,k}.
            \]
            In addition we require $s_1>a$ when $k>1$.
        \item \textbf{Balance}:
            For each $i\in[n_k]$
            \[
                |E_k\cap (s_{i,k},t_{i,k})| \geq 0.1(t_{i,k}-s_{i,k}).
            \]
            with equality holding unless $i=n_k$ and $t_i=1$.
    \end{itemize}
    We call the elements of $\mathcal{I}^J_k$ \emph{canonical intervals}
    for the index $k$.
\end{definition}
Notice that the balance condition enforces that either $n_k=1$ and
$E_k=(s,1)$, or else
\[
    E_k \subsetneq \bigcup_{i=1}^{n_k} (s_{i,k},t_{i,k}).
\]

It is not difficult to see that a canonical collection always exists, and
we choose now a canonical collection $\mathcal{I}$.
We write $\mcal{I}_k = \{(s_{i,k},t_{i,k})\}_{i=1}^{n_k}$ for the collection
of canonical intervals covering $E_k$.

One way to think about
$\mathcal{I}_k$ is that it describes a simplification of the level set $E_k$ that
is more robust to the kinds of rapid oscillations that $g$ may have which
ultimately do not affect the nonlocal integral.  A key concept that uses the
collection $\mathcal{I}$ is that of a compatible interval.
\begin{definition}
    An interval $I$ is \emph{compatible} with $\mathcal{I}_k$ if for
    every $(s,t)\in\mathcal{I}_k$, either $(s,t)\subset I$ or $(s,t)$ is
    disjoint from $I$.
\end{definition}

Another important use of the idea of canonical intervals is that
Proposition~\ref{rigidity-lemma} allows us to define a dichotomy between
bottom-heavy intervals and top-heavy intervals.
\begin{definition}[Bottom-heavy and top-heavy intervals]
    Let $I\in\mathcal{I}_k$ be a canonical interval with
    $|E_k\cap I|=0.1|I|$.
    We say that $I$ is a \emph{bottom-heavy} interval if $|I\cap E_{k+1}| > 0.8|I|$.
    If on the other hand $|I\cap E_{k-1}| > 0.8|I|$, then we say that $I$ is
    a \emph{top-heavy} interval.
\end{definition}

If $k\in P$ and $I\in\mathcal{I}_k$, then $I$ must be either top-heavy,
bottom-heavy, or else $|I\cap E_k| > 0.1|I|$.    By the balance condition,
the latter case can only happen when the right-endpoint of $I$ is $1$.  It
will be useful to introduce some terminology to handle this case.

\begin{definition}[Terminal intervals]
    A \emph{terminal} interval for the index $k$ is an interval
    $I=(a,b)\subset [0,1]$ such that $|I\cap E_k| > 0.03|I|$ and
    either $a=0$ or $b=1$.  If $a=0$, then $I$ is a \emph{left-terminal} interval.
    Conversely if $b=1$ then $I$ is a \emph{right-terminal} interval.
    A \emph{maximal} terminal interval $I$ is a terminal interval
    with the property that any interval $I'$ with $I\subset I'$
    satisfies $|I'\cap E_k| < 0.03|I|$.
\end{definition}

The following simple lemma is one main reason we can hope to
demonstrate the existence of a logarithmic divergence.
\begin{lemma}
    Let $|k-k'|>1$ and either $k\not\in\mathcal{S}$ or
    $k'\not\in\mathcal{S}$.  Suppose $k$ and $k'$ have maximal left-terminal
    intervals of length $\tau$ and $\tau'$, respectively.  Then either
    $\tau>3\tau'$ or $\tau'>3\tau$.  The same statement holds for maximal
    right-terminal intervals.
    \label{interval-growth}
\end{lemma}
\begin{proof}
    Without loss of generality suppose that $\tau>\tau'$ but $\tau<3\tau'$.
    Consider the interval $I=[0,3\tau']$.  Observe that
    $|I\cap E_k|,|I\cap E_{k'}| > 0.01|I|$, which is a contradiction.
\end{proof}

We now take another small step towards demonstrating logarithmic
divergence, which is to define a set of indices and
lengths $\mathcal{G}_L =\{(k_i, \tau_i)\}_{i=1}^m$
which we expect to serve as a witness to logarithmic growth near the
left endpoint $0$.

The set $\mathcal{G}_L$ is defined inductively as follows.  We set
$k_0=0$ and $\tau_0$ to be the length of the maximal left-terminal
interval for the index $0$.\footnote{Observe that $0$ must have a
    left-terminal interval because $\theta(0)=0$.}
Given $k_i$ and $\tau_i$, let $k$ be the least index
$k\in \mathcal{P}$ such that $k>k_i$ and
such that there is a maximal left-terminal
interval $(0,\tau)$ with $\tau > 3\tau_i$.  If such an index exists,
we set $k_{i+1}=k$ and $\tau_{i+1}=\tau$ and continue.
If not, the process halts and the definition of $\mathcal{G}_L$ is
complete.

We also define $\mathcal{G}_R$, the set of indices that are separated from
$\mathcal{G}_L$ and are not in $\mathcal{G}_L$.  More precisely,  define
\[
    \mathcal{G}_R = \{k\in \mathcal{P}\setminus \mathcal{G}_L
    \mid  k-1\not\in \mathcal{G}_L\}.
\]
We have used notation to suggest that $\mathcal{G}_R$ is a good candidate
for a witness of logarithmic growth near the right endpoint $1$.  This
is not clear from the definition, and in fact the largest remaining task
is to prove that this is the case.

Our first main lemma sets up conditions under which we can control the size
of the discretized superlevel sets
\[
    E_{>k} := \{t\in [0,1] \mid g(t) > (k+1)\delta\}
\]
using only information about the index $k$.  In particular we introduce
the notion of a grounded interval.
\begin{definition}[Grounded interval]
    Let $J=(a,b)\subset[0,1]$ be an interval.  We say that $J$ is
    \emph{grounded} for the index $k$ if $J$ is an interval compatible with
    $\mathcal{I}_k$, if $g(a)<(k+1)\delta$, and if for every $I\in \mcal{I}_k$
    with $I\subset J$, $I$ is a bottom-heavy interval.
\end{definition}

\begin{lemma}
    Let $k\in \mathcal{P}$, and let $J=(a,b)\subset[0,1]$ be a grounded
    interval for the index $k$.  Then
    \begin{equation}
        |E_{>k}\cap J| \leq \frac{1}{5}|E_k\cap J|
        \label{bh-bd}
    \end{equation}
    and
    \begin{equation}
        |E_{k}\cap J| < \frac{1}{5}|E_{k-1}\cap J|.
        \label{bh-bd2}
    \end{equation}
    \label{main-bh-lemma}
\end{lemma}
\begin{proof}
    Let $I_1,I_2,\cdots,I_n$ with $I_i=(s_i,t_i)$ be an enumeration in order
    of all bottom-heavy canonical intervals in $\mathcal{I}_k$ that are
    contained in $J$.   The proof goes by induction on $n$.
    In the case $n=0$, the covering property implies $E_k\cap J$ is empty.
    Therefore $\theta(a) \leq k\delta$ and the intermediate value
    theorem implies that $\theta(s)\leq k\delta$ for all $s\in J$, and
    thus $|E_{>k}\cap J| = 0$.

    Now consider the case $n\geq 1$, and let $I_1=(s_1,t_1)$.
    Again by the intermediate value theorem and the covering
    property of $\mathcal{I}_k$, $\theta(s)\leq k\delta$ for $s\in (a,s_1)$.
    Since $k\in \mathcal{P}$ and $I_1$ is bottom-heavy, we have
    $|I_1\cap E_k| = 0.1|I_1|$, $|I_1\cap E_{>k}|<0.01|I_1|$, and
    $|I_1\cap E_{k-1}| > 0.8|I_1|$.  If $\theta(t_1)\leq (k+1)\delta$,
    then the interval $(t_1,t)$ satisfies the inductive hypothesis, having
    $n-1$ canonical intervals inside.

    The remaining case is that $\theta(t_1)> (k+1)\delta$.  Let $t'$ be
    the first time $t'>t_1$ for which $\theta(t')=(k+1)\delta$.  Suppose
    that $t'>(t_1-s_1)$, and consider
    the interval $I_1' = (s_1, 2t_1-s_1)$.  In this case,
    $|N_{k-1}^c\cap I_1'| > 0.5|I'_1|$ and
    $|E_{k-1}\cap I_1'| > 0.4|I_1'|$, which
    contradicts the assumption that $k-1\not\in \mathcal{S}$.
    Thus it must be that $t' < (t_1-s_1)$.

    Define now the interval $I=(s_1,t')$.
    Then $|E_{k-1}\cap I| > 0.4 |I|$, so
    it follows that $|N_{k-1}^c\cap I| < 0.01|I|$, and in
    particular
    \[
        |E_{>k}\cap I| < 0.01|I| < 0.02|I_1| =
        \frac{1}{5}|E_k\cap I_1|.
    \]
    At this point one can
    now induct, using the smaller interval $I' = [t',t]$ which has fewer
    bottom heavy canonical intervals.
\end{proof}

The following lemma is our main tool for finding grounded intervals.

\begin{lemma}
    Let $m\in \mathcal{G}_R$ and
    $k\in \mathcal{G}_L$ be the largest index in $\mathcal{G}_L$
    satisfying $k<m$.  Let $\tau_k$ be the length of the maximal
    left-terminal interval for $k$.
    Then there exists an interval $J=(s,t)$ such that $s<\tau_k/2$,
    $J$ is grounded for the
    index $m$, and $(t,1)$ is compatible with $\mathcal{I}_m$.
    \label{bh-Ac-lemma}
\end{lemma}
\begin{proof}
    Observe that, by the definition of $\mathcal{G}_R$, we have $k<m-1$.

    We directly define
    \[
        s = \inf\{s'\in (\tau_k/3,1] \mid g(s') < (m+1)\delta,
            s'\not\in \bigcup_{I\in \mcal{I}_m\cup \mcal{I}_{m+1}} I\}.
    \]
    First, we claim that $s < \tau_k/2$.  Indeed, for every point
    $s'\in (\tau_k/3,s)$, we either have
    $\theta(s) > (m+1)\delta$ or else there exists
    an interval $I\in\mathcal{I}_m\cup\mathcal{I}_{m+1}$ such that
    $I\subset[0,s)$ and $s'\in I$.
    Thus
    \[
        |(0,s)\cap E_{\geq m}| \geq
        \sum_{\substack{I\in\mathcal{I}_m\cup\mathcal{I}_{m+1} \\
        I\subset (0,s)}} 0.1|I| +
        \Big|(\tau_k/3,s)\setminus
        \bigcup_{I\in\mcal{I}_m\cup\mcal{I}_{m+1}} I\Big| \geq
        0.1 (s-\tau_k/3)
    \]

    If $s>\tau_k/2$, then this implies
    \[
        |(0,\tau_k)\cap E_{\geq m}| \geq
        |(0,\tau_k/2)\cap E_{\geq m}| \geq \frac{1}{60}\tau_k,
    \]
    this would contradict $k\not\in \mathcal{S}$,
    since $|(0,\tau_k)\cap E_k| > 0.03\tau_k$
    and $k< m-1$.

    It remains to show that every $I\in\mathcal{I}_m$ with
    $|I\cap (s',1)|>0$ is either bottom-heavy or a terminal interval.

    On the contrary, suppose that $I=(s_0,s_1)$
    is the first top-heavy interval in $\mathcal{I}_m$
    with $s_0>s$.  Consider $\sigma = s_0 - 2(s_1-s_0)$.  If $\sigma<0$,
    then
    \[
        |E_{m}\cap (0,s_1)| \geq 0.1 (s_1-s_0) \geq 0.03 (s_1-\sigma).
    \]

    In particular $(0,s_1)$ is a left-terminal interval for $m$ with
    $s_1 > \tau_k/3$, which along with Proposition~\ref{interval-growth}
    contradicts $m\not\in \mathcal{G}_L$.

    Therefore we can assume that $\sigma>0$.  The next sub-case to consider
    is $\sigma < s$.  In this case,
    split the interval $(\sigma,s_0)$ into the interval $(\sigma,s)$ and $(s,s_0)$.
    Since $s\not\in I$ for
    any $I\in\mathcal{I}_m\cup \mathcal{I}_{m+1}$,
    we have $|(\sigma,s)\cap E_m| < 0.1 (s-\sigma)$ and
    $|(\sigma,s)\cap E_{m+1}| < 0.1 (s-\sigma)$.  Moreover,
    $|(s,s_0)\cap E_{m}| < 0.1(s_0-s')$ because every canonical interval
    that intersects $(s,s_0)$ is contained in $(s,s_0)$.  Finally,
    $|(s,s_0)\cap E_{m+1}| < 0.02(s_0-s)$ because the interval $(s,s_0)$
    satisfies the conditions of Lemma~\ref{main-bh-lemma}.  In conclusion, we have
    \[
        |E_m\cap (\sigma,s_0)| < 0.1 (s_0-\sigma)
    \]
    and
    \[
        |E_{m+1}\cap (\sigma,s_0)| < 0.1 (s_0-\sigma).
    \]
    These inequalities are mutually incompatible with the hypotheses that
    $m,m+1\not\in \mathcal{P}$, since $|E_m\cap (\sigma,s_1)| > 0.03(s_1-\sigma)$
    and $|E_{m+1}\cap (\sigma,s_1)| > 0.1(s_1-\sigma)$ because $(s_0,s_1)$ is top-heavy.

    The final case to consider is $\sigma>s$.  If $(\sigma,s)$ is compatible with
    $\mcal{I}_k$, then the argument above applies to reach a contradiction.  The remaining
    case therefore is that $\sigma\in (t_0,t_1)\in \mcal{I}_k\cup\mcal{I}_{k+1}$.
    If $(t_0,t_1)\in \mcal{I}_k$, then because $(t_0,t_1)$ is bottom-heavy by hypothesis,

\end{proof}

One consequence of Lemma~\ref{bh-Ac-lemma} and Lemma~\ref{main-bh-lemma} is
that the sets $\mathcal{G}_L$ and $\mathcal{G}_R$ cannot interlace much.
\begin{lemma}
    Let $b\in \mathcal{G}_R$
    and suppose that there exists $k\in \mathcal{G}_L$ such that $k>b$.
    Then the length $\tau_k$
    of the maximal left-terminal interval for $k$
    satisfies $\tau_k>\frac{1}{3}$.
    In particular, $k$ is the largest element in $\mathcal{G}_L$.
    \label{no-gaps}
\end{lemma}
\begin{proof}
    Let $a\in \mathcal{G}_L$ be the largest index in $\mathcal{G}_L$
    with $a<b$, and let $\tau_a$ be the length of the maximal
    left-terminal interval for $a$.

    We will first show that $\tau_k$ is in a terminal interval of
    $\mathcal{I}_b$, by first showing that $\tau_k$ is in a canonical
    interval of $\mathcal{I}_b$, and then by demonstrating that it
    is terminal.

    Applying Lemma~\ref{bh-Ac-lemma}, we may find an interval
    $(s,t)$ with $s<\tau_a/2$ which is grounded for the index $b$.
    If $\tau_k\not\in I$ for every $I\in\mathcal{I}_b$, then
    the interval $(s,\tau_k)$ is grounded for the index $b$.
    Then by Lemma~\ref{main-bh-lemma},
    \[
        |E_k\cap (s,\tau_k)| \leq
        |E_{>b}\cap (s,\tau_k)| < \frac{1}{5}|E_b\cap (s,\tau_k)| =
        0.02(\tau_k-s).
    \]
    Moreover since $a\not\in \mathcal{S}$,
    \[
        |E_k\cap (0,s)| < 0.01 s.
    \]
    Combined, these give $|E_k\cap (0,\tau_k)| < 0.02\tau_k$,
    so $\tau_k$ cannot be a left-terminal interval.

    Now we check that $\tau_k$ must be contained in a terminal interval.
    If not, then $\tau_k\in (s_b,t_b)\in\mathcal{I}_b$
    with $t_b<1$.  The same argument used above
    shows that $|E_k\cap (0,s_b)| < 0.02s_b$.  We will now prove that
    $|E_k\cap(s_b,\tau_k)| < 0.02(\tau_k-s_b) + 0.01\tau_k$, which will
    suffice to arrive at the contradiction $|E_k\cap(0,\tau_k)|<0.03\tau_k$.

    First, using the fact that $t_b<1$
    and $b\not\in \mathcal{G}_L$, we have
    \[
        0.1(t_b-s_b) = |E_b\cap (s_b,t_b)| \leq |E_b\cap(0,t_b)| < 0.03t_b.
    \]
    Rearranging and using $s_b<\tau_k$, we can certainly achieve the bound
    $t_b<2\tau_k$.

    We combine this with the bound $|E_k\cap(s_b,t_b)|<0.01(t_b-s_b)$ which
    comes from the fact that $J$ is bottom heavy, to conclude
    \begin{align*}
        |E_k\cap(s_b,\tau_k)|
        \leq  |E_k\cap (s_b,t_b)|
        &\leq  0.01(t_b-s_b) \\
        &< 0.01(2\tau_k - s_b) \leq 0.02(\tau_k-s_b) + 0.01\tau_k,
    \end{align*}
    as desired.

    We have thus demonstrated that $\tau_k \in (s_b,1)\in\mathcal{I}_b$.
    Since $\tau_k < 1$, it must follow that $\tau_a < \frac{1}{3}$.
    This allows us to write
    \[
        0.1 (1-s_b) < |E_b\cap(s_b,1)| < |E_b\cap(0,1)| < 0.03,
    \]
    from which we conclude $s_b > 0.7 > \frac{1}{3}$.  Since $\tau_k>s_b$
    this certainly implies the conclusion.
\end{proof}

Let $\tau_L$ be the length of the largest maximal left-interval in
$\mathcal{G}_L$, and $\tau_R$
be the largest maximal right-interval in
$\mathcal{G}_R$ ($\tau_R=0$ is possible).
The thrust of the previous two lemmas is that we achieve exponential decay
for the maximal right-intervals of the indices in $\mathcal{G}_R$.

\begin{lemma}
    Let $m\in \mathcal{G}_R$,
    and let $k>m+1$.  Let $\tau_m$ be the length of the
    maximal right-interval for $m$.  Then the length $\tau_k$ of the
    maximal right-interval for $k$ satisfies either $3\tau_k < \tau_m$ or
    $\tau_k > \frac{1}{2}$.
    \label{decaying-r-lemma}
\end{lemma}
\begin{proof}
    By Proposition~\ref{interval-growth}, we only need to rule out the
    case $3\tau_m < \tau_k < \frac{1}{2}$.  By Lemma~\ref{bh-Ac-lemma},
    there exists an interval $J$ containing $(\frac{1}{2},1-\tau_m)$ which
    is grounded for index $m$.  If $\tau_k < \frac{1}{2}$, then
    $1-\tau_k \in (\frac{1}{2},1-\tau_m)$.  Let $s\in J$ be maximal such
    that $(s,1-\tau_m)$ is a grounded interval and $s\leq 1-\tau_k$.
    Observe that either $s=1-\tau_m$ or else $(s,t)\in\mathcal{I}_m$
    for some $t > 1-\tau_k$.  In either case we have
    \[
        |E_m\cap (s,1)| > 0.1 (1-\tau_k-s).
    \]
    Combined with the inequality
    \[
        |E_m\cap (s,1)| < 0.03 (1-s),
    \]
    which follows from the maximality of $\tau_m$, we obtain
    \[
        1-s < \frac{10}{7} \tau_k.
    \]
    Next, because $(s,1-\tau_m)$ is a grounded interval, it follows from
    Lemma~\ref{main-bh-lemma} that
    \[
        |E_k\cap (1-\tau_k,1-\tau_m)|
        \leq |E_k\cap (s,1-\tau_m)|
        \leq \frac{1}{5}|E_m\cap(s,1-\tau_m)|
        < 0.03 \frac{2}{7} \tau_k.
    \]
    Now we compute, using also $\tau_k > 3\tau_m$,
    \begin{align*}
        |E_k\cap(1-\tau_m,1)|
        &= |E_k\cap(1-\tau_k,1)| - |E_k\cap(1-\tau_k,1-\tau_m)|
        \\&\geq 0.03\tau_k(1 - \frac{2}{7}) > 0.01\tau_m.
    \end{align*}
    This violates the fact that $m\not\in\mathcal{S}$, so we have arrived
    at a contradiction.
\end{proof}

We need one final lemma, which is useful in the case that
$\#\mathcal{G}_L < 0.99M$.

\begin{lemma}
    If $k\in \mathcal{G}_R$ and $k\leq M$,
    then there is a right-terminal interval for $k$.
    \label{end-interval-lemma}
\end{lemma}
\begin{proof}
    By Lemma~\ref{bh-Ac-lemma}, there exists $s$ and $t$
    such that $(s,t)$ is a grounded interval for $k$, and $(t,1)$ is a
    terminal interval.

    In particular, there exist canonical intervals $J\in\mathcal{I}_k$
    contained in $(s',1)$.  Consider the last such interval
    $J = (s_{n_k,k}, t_{n_k,k}) =: (s_f,t_f)$.  Our claim is that
    $(s_f,1)$ is a right-terminal interval.  Indeed, if
    $t_f = 1$, this means that $|J\cap E_k| > 0.1|J|$, and we are done.

    Otherwise, we have that $|J\cap E_k| = 0.1|J|$ and
    $|J\cap E_{k-1}| > 0.8|J|$, because $J$ must be bottom-heavy.
    In this case, it suffices to show that $1-t_f < (1-s_f)/2$.

    The hypothesis that $J$ is the final canonical interval and
    $k<\delta^{-1}D$ together imply that $\theta(t) > (k+1)\delta$ for all
    $t>t_f$.  Suppose that $t_f < 1-\tau/2$, and consider the interval
    $J' = (1-\tau, 1-\tau/2) \subset [0,1]$.  Observe that
    $|J'\cap E_{k-1}| > 0.35(\tau/2)$ and
    $|J'\cap L_{k-1}^c| > 0.5 (\tau/2)$.
    This would contradict the hypothesis that $k-1\not\in \mathcal{S}$.

    In particular, the interval $I=(s_f,1+(t_f-s_f))\subset[0,2]$ satisfies
    $|I\cap E_k| > 0.01|I|$ and
    \[
        |I\cap E_{>N}| = 0.5|I|.
    \]
\end{proof}

We are now ready to complete the proof of the main result.
\begin{proof}[Proof of Theorem~\ref{structure-thm}]
    We can first assume that $\#\mathcal{S} < 0.01M$, in which case
    we need to demonstrate a logarithmic singularity for $g$.
    As we observed, this assumption implies that
    $\#\mathcal{P}^c \leq 0.03M$.

    We now split the analysis into two cases.  In the first case,
    we suppose that $\#\mathcal{G}_L < 0.9M$.  In this case,
    one has $\#(\mathcal{G}_R\cap [M-1]) > 0.07M$.  Consider the set
    \[
        \mathcal{W} = \{k\in \mathcal{G}_R\cap [M] \mid
        \tau_k < \frac{1}{2} \text{ and } k \text{ is even}\},
    \]
    Replacing `even' with `odd' if necessary, we may assume that
    $\#\mathcal{W} > 0.03M$
    (provided that $M$ is sufficiently large, because at most $100$
    indices can have $\tau_k > \frac{1}{2}$).  Using
    Lemma~\ref{decaying-r-lemma}, we see that
    $\tau_k > 3\tau_{k'}$ when $k'>k$ and $k,k'\in\mathcal{W}$.
    Moreover, as $E_{k-M}$ is empty for each $k\in\mathcal{W}$, we see
    that the indices in $\mathcal{W}$ serve as a witness for a logarithmic
    singularity of $g$ at $1$.

    In the second case, we have $\#\mathcal{G}_L > 0.9M$.  Let
    $\ell$ be the second-largest index in $\mathcal{G}_L$.  By
    Lemma~\ref{no-gaps}, $\mathcal{G}_R\cap[\ell]$ is empty.  Thus
    the set
    \[
        \mathcal{W} = \{k\in \mathcal{G}_L \mid
        k+M \geq \ell+1\text{ and } k,k+M \in \mathcal{P}\}
    \]
    satisfies $\#\mathcal{W} \geq 0.45 M$.  By discarding at most
    two elements of $\mathcal{W}$ and then filtering to odd or
    even indices, we can arrange that $\mathcal{W} > 0.2M$ and that
    $k+M\in\mathcal{G}_R$ for every $k\in\mathcal{W}$.

    Now enumerate the elements of $\mathcal{W}$ in order, writing
    \[
        k_0 < k_1 < \cdots < k_m
    \]
    where $m=\#\mathcal{W}$.  Let $\tau^L_i$ be the length of the
    maximal left-terminal interval for the index $k_i$, and let
    $\tau^R_i$ be the length of the maximal right-terminal interval
    for the index $k_i+M$.  Then by the definition of
    $\mathcal{G}_L$,
    \[
        \tau_0^L < \frac{1}{3}\tau^L_1
        < \cdots < \frac{1}{3^k}\tau^L_i < \cdots < \frac{1}{3^m}\tau^L_m.
    \]
    Moreover, by Lemma~\ref{decaying-r-lemma},
    \[
        \tau_0^R > 3\tau_1^R > \cdots > 3^k\tau^R_i > \cdots > 3^m\tau^R_m.
    \]
    Suppose that for some index $k_i$,
    $\tau_i^L > 30\tau_i^R$.  Then $(0,\tau_i^L)$ is $0.03$-full for
    $k_i$ and $(1-\tau_i^L,1)$ is $0.03$-empty for $k_i+M$.
    On the other hand if $\tau_i^L < \frac{1}{30}\tau_i^R$ then
    $(1-\tau_i^R,1)$ is $0.03$-full for $k_i+M$ and $0.03$-empty
    for $k_i$.

    There can only be at most, say, $20$ indices for which both of the
    above hypotheses fail.  Then splitting into the most common case, we
    either have a witness for a singularity at $0$ or at $1$.
\end{proof}

\bibliographystyle{alpha}
\bibliography{refs}

\end{document}